\documentclass[12pt]{amsart}
\usepackage{amsmath,latexsym,amsmath}
\usepackage{amsfonts}
\usepackage{amssymb}
\usepackage{geometry}
\usepackage{graphicx}
\usepackage[greek,english]{babel}
\usepackage{amsthm}
\usepackage{bm}
\usepackage{setspace}
\usepackage[final]{showkeys}
\usepackage{dsfont}

\theoremstyle{definition}
\newtheorem{The}{Theorem}

\newtheorem{Pro}[The]{Proposition}

\newtheorem{Rem}[The]{Remark}
\begin{document}

\title[Large deviations for piecewise expanding maps]{A note on the large deviations for piecewise expanding multidimensional maps}
%\author{N. Haydn, M. Nicol, T. Persson and S. Vaienti}
\author{R. Aimino}
\address{R. Aimino\\ UMR-6207 Centre de Physique Th\'eorique, CNRS,
Universit\'es d'Aix-Marseille, Universit\'e du Sud, Toulon-Var and
FRUMAM, F\'ed\'ederation de Recherche des Unit\'es de
Math\'ematiques de Marseille\\ CPT, Luminy Case 907, F-13288
Marseille Cedex 9, France} \email {romain.aimino@orange.fr}
\author{ S. Vaienti}
\address{S. Vaienti\\ UMR-6207 Centre de Physique Th\'eorique, CNRS,
Universit\'es d'Aix-Marseille, Universit\'e du Sud, Toulon-Var and
FRUMAM, F\'ed\'ederation de Recherche des Unit\'es de
Math\'ematiques de Marseille\\ CPT, Luminy Case 907, F-13288
Marseille Cedex 9, France} \email {vaienti@cpt.univ-mrs.fr}

\thanks{SV warmly thanks E. Ugalde for the kind invitation to participate to the Conference in honor
of Valentin Afraimovich; RA and SV express their sincere gratitude
to I. Melbourne who helped them to simply the proofs.}

\begin{abstract}
We provide the large deviation principle for higher dimensional
piecewise expanding maps and by using the functional approach of
Hennion and Herv\'e, slightly modified.
\end{abstract}

\maketitle

\section{Introduction}

There are different ways to establish large deviation principles
(LDP) for dynamical systems. One of them is the so-called "Laplace
Method", which relies on the spectral properties of the
Perron-Frobenius, or transfer, operator. This strategy has been
developed in a very general and abstract setting by Hennion and
Herv\'e in \cite{HH}. They assume  that the transfer operator acts
on a Banach spaces of measurable functions, and it is quasi-compact
on it, i.e. it  has a spectral gap. The existence of an invariant
probability measure follows immediately, and, by using perturbation
theory for linear operators, they derive a few others statistical
properties. This approach covers a lot of systems, for instance
expanding maps of the interval \cite[XII.1]{HH}, Gibbs measures for
subshift of finite type \cite[XII.2]{HH}, and  expanding Young
towers \cite{LSY1}. Nevertheless, this theory seems  inappropriate
for expanding (discontinuous)  maps in higher dimension, like those
treated  by several authors \cite{Ad, Bl, BG, Bu, BK, C, Liv2, PGB, S, Th, Tho,
Ts}. Indeed, Hennion and Herv\'e assume that the Banach space on
which  the transfer operator acts,  consists of measurable functions
defined everywhere, but for the higher dimensional systems quoted
above, the functional spaces usually considered (bounded variation,
quasi-H\"older or Sobolev spaces) consist of classes of equivalence
of functions modulo the reference measure, and hence they are only
defined almost everywhere. Furthermore, in \cite{HH}, the Dirac
masses must belong to the topological dual of the Banach space, so
this theory cannot be applied directly for those systems.

Nevertheless it appears that one could  slightly modify the proofs
from \cite{HH} in order to deal with a Banach space consisting of
classes of functions. In particular we will do it for the functional
space of quasi-H\"older functions mostly investigated in \cite{S}
and which verify an additional algebraic assumption which also plays
a role in the Hennion-Herv\'e approach. As a consequence we will get
the large deviations principle for such systems and, as far as we
know, this result is not present in the literature. We also prove
the central limit theorem, but for the latter one already disposes
of the  Gordin-Liverani theorems \cite{Go, Liv}.
\\Actually a weaker result for the large
deviations of systems like those considered above has been recently
obtained in \cite{AFLV}. We will comment about the difference with
the spectral technique presented in this note in the Remark 2 below.
We anticipate here that the paper \cite{AFLV} furnishes an {\em
upper} bound for the deviation functions and whenever the
correlation functions involving $L^1$ observables decay to zero with
a summable rate. In order to check these assumptions for our systems
we should further require that the density of the invariant measure
is essentially bounded from below, but this assumption  is not
necessary in the spectral approach discusses later on.
\section{Assumptions and statement of the results}

We  now give the precise assumptions under which the LDP is valid.
Let $(X, \mathcal{A}, m)$ a probability space, and $T : X \to X$ a
measurable transformation, non singular with respect to $m$.
 Under these conditions, the
Perron-Frobenius operator $P : L^1(m) \to L^1(m)$ is well defined by
$Pf = \frac{d m_f}{d m}$, where $m_f(A) =
\int_{T^{-1}A} f \, dm$ is absolutely continuous with respect to
$m$. We stress here the fact that the functions under consideration
are complex valued, as required by the  spectral theory we are going
to use below. The transfer operator enjoys some classical properties
that we resume below; see \cite{BoGo} or \cite{LM} for more details.

\begin{enumerate}
\item {\em Linearity}   : $P$ is a linear operator on $L^1(m)$, satisfying $||Pf||_1 \le ||f||_1$ for all $f \in L^1(m)$;
\item {\em Positivity} : For all $f \in L^1(m)$ such that $f \ge 0$ $m$-ae, we have $Pf \ge 0$ $m$-ae;
\item {\em Preservation of integrals} : For all $f \in L^1(m)$, we have $\int Pf \, dm = \int f \, dm$;
\item {\em Duality} : For all $f \in L^1(m)$ and $g \in L^{\infty}(m)$, we have $\int f (g \circ T) \, dm = \int (Pf) g \, dm$;
\item {\em Invariant Measures} : $f \in L^1(m)$ is the density of a $T$-invariant probability if and only if $f \ge 0$, $\int f \, dm = 1$ and $Pf = f$.
\end{enumerate}

 ~ \\ \noindent Let us suppose now that we have a subspace $\mathcal{B} \subset L^1(m)$, equipped with a norm $|| \, . \, ||_{\mathcal{B}}$ such that
\begin{enumerate}
\item $(\mathcal{B}, || \, . \, ||_{\mathcal{B}})$ is a complex Banach space with continuous injection $\mathcal{B} \to L^1(m)$;
\item Constant functions lie in $\mathcal{B}$;
\item $\mathcal{B}$ is a Banach algebra : there exists $C > 0$ such that for all $f, g \in \mathcal{B}$ we have $fg \in \mathcal{B}$ with $||fg||_{\mathcal{B}} \le C ||f||_{\mathcal{B}} ||g||_{\mathcal{B}}$;
\item $\mathcal{B}$ is a complex Banach lattice : for every $f \in \mathcal{B}$, we have $\bar{f}, |f| \in \mathcal{B}$;
\item $\mathcal{B}$ is stable under $P$ : $P(\mathcal{B}) \subset \mathcal{B}$;
\item $P$ is a bounded operator on $\mathcal{B}$, with spectral radius equal to one;
\item $P$ is quasi-compact of diagonal type on $\mathcal{B}$.
\end{enumerate}

The last assertion means that there exists a decomposition $$P =
\sum_{i=1}^s \lambda_i \Pi_i + Q$$ where $\lambda_i$ are complex
numbers of modulus $1$, $\Pi_i$ are finite-rank projections
satisfying $\Pi_i \Pi_j = 0$ when $i \neq j$ and $Q$ is a bounded
operator on $\mathcal{B}$ with spectral radius strictly less than
$1$ and satisfying $Q \Pi_i = \Pi_i Q = 0$ for all $i$. The spectrum
of $P$ consists then of a finite number of eigenvalues of modulus
$1$, with finite multiplicity, and the rest of the spectrum lies in
a disc centered at $0$ with radius strictly less than $1$. When
$\mathcal{B}$ is compactly injected in $L^1(m)$, this can be deduced
from a Lasota-Yorke type inequality, by means of the Ionescu-Tulcea
and Marinescu theorem \cite{ITM, H}. See \cite{HH} for precise
definitions and results about quasi-compactness.

Under those conditions, the existence of an $T$-invariant
probability $\mu$ absolutely continuous w.r.t $m$, such that
$\frac{d \mu}{dm} \in \mathcal{B}$ is a classical result : for every
$f \in \mathcal{B}$ such that $f \ge 0$, $\int f \, dm = 1$, and so
in particular for $f = \mathds{1}$, quasi-compactness implies that
the sequence $\frac{1}{n} \sum_{k=0}^{n-1} P^k f$ converges in
$\mathcal{B}$ to a function $f^{\star}$ such that the measure $\mu_f$ with $\frac{d\mu_f}{dm} = f^{\star}$ is an
acip. Furthermore, $1$ is an eigenvalue of $P$. If we assume that
$1$ is a simple eigenvalue of $P$, then there exists an unique acip
$\mu$ such that $\frac{d \mu}{ dm} \in \mathcal{B}$. From now, we
will always assume that $1$ is a simple eigenvalue, and that there
is no other eigenvalue of modulus $1$. $\mu$ will denote the unique
acip, and $v \in \mathcal{B}$ its density. We then have
\footnote{When $\varphi \in \mathcal{B}^{\star}$ belongs to the
topological dual of $\mathcal{B}$, we denote $<\varphi, f> =
\varphi(f)$. The linear form $f \to \int f \, dm$ belongs to
$\mathcal{B}^{\star}$, and we denote it by $m$.}, for all $n \ge 1$
and $f \in \mathcal{B}$ $$P^nf = <m,f> v + Q^n f$$ As a consequence,
we get exponential decay of correlation : there exists $C \ge 0$ and
$0 \le \lambda < 1$ such that for every $f \in \mathcal{B}$ and
every $g \in L^{\infty}(\mu)$ we have $$\left\vert \int f(g \circ
T^n) \, d \mu - \int f \, d \mu \int g \, d \mu \right\vert \le C
\lambda^n ||f||_{\mathcal{B}} ||g||_{L^{\infty}_{\mu}}$$

Let now $\phi : X \to \mathbb{R}$ a bounded observable which lie in
$\mathcal{B}$, with zero mean $\int \phi \, d \mu = 0$. Denote by
$S_n$ the Birkhoff sums : $$S_n = \sum_{k=0}^{n-1} \phi \circ T^k$$
We are now able to state the LDP :

\begin{The} (Large Deviation Principle)
\em

\noindent Under the above conditions, the limit $\sigma^2 = \lim_{n
\to \infty} \int ( \frac{S_n}{\sqrt{n}} )^2 \, d \mu$ exists, and if
$\sigma^2 > 0$, then there exists for some $\epsilon_0 > 0$ a rate
function $c : \,  ]-\epsilon_0, + \epsilon_0[ \to \mathbb{R}$,
continuous, strictly convex, vanishing only at $0$, such that for
every $0 < \epsilon < \epsilon_0$ and every probability measure $\nu$ with $\nu \ll m$ and $\frac{d \nu}{dm} \in \mathcal{B}$, we have $$\lim_{n \to \infty}
\frac{1}{n} \log \nu (S_n > n \epsilon) = - c(\epsilon)$$
\end{The}
As an easy consequence of the techniques introduced in the next section, we
also get the central limit theorem. We denote with $\mathcal{N}(0,
\sigma^2)$ the Dirac mass $\delta_0$ if $\sigma^2 = 0$, and the
probability with density $\frac{1}{\sigma \sqrt{2\pi}} e^{-
\frac{t^2}{2 \sigma^2}}$ with respect to Lebesgue if $\sigma^2
> 0$.

\begin{The} (Central Limit Theorem)
\em

\noindent Under the same assumptions  of Theorem 1,
$\frac{S_n}{\sqrt{n}}$ converges in distribution to $\mathcal{N}(0,
\sigma^2)$ in the probability space $(X, \mathcal{A}, \nu)$ for every probability
$\nu$ with $\nu \ll m$ and $\frac{d\nu}{dm} \in \mathcal{B}$ : for
every bounded continuous function $g : \mathbb{R} \to \mathbb{R}$,
we have
$$\lim_{n \to \infty} \int g(\frac{S_n}{\sqrt{n}}) d\nu = \int g
\, d \mathcal{N}(0, \sigma^2)$$

\end{The}
\begin{Rem} ~
\begin{enumerate}
\item Theorems 1 and 2 apply in particular for $\nu = m$ and $\nu = \mu$, so the LDP and the CLT are valid for both reference and invariant measures.
\item As we anticipated in the Introduction,  the paper \cite{AFLV} gives  an upper bound for the large deviation function and  under related
assumptions. In particular Th. E in \cite{AFLV} states the
following, with our notations. Let us suppose that  $T$ preserves an
ergodic probability measure $\mu$; then let $\mathcal{B} \subset
L^1(\mu)$, $\phi\in \mathcal{B}$,  and assume that there exists
$\xi(n)$ with $\sum_{n=0}^{\infty}\xi(n)<\infty$ such that for all
$\psi\in L^1(\mu)$ we have $\left\vert \int \phi(\psi \circ T^n) \,
d \mu - \int \phi \, d \mu \int \psi \, d \mu \right\vert \le \xi(n)
||f||_{\mathcal{B}} ||g||_{L^{\infty}_{\mu}}$. Then there exists
$\tau=\tau(\phi)>0$ and, for every $\epsilon>0$, there exists
$C=C(\phi,\epsilon)>0$ such that $\mu(S_n> n\epsilon)\le Ce^{-\tau
n}$.

The proof of this result relies on a martingale approximation.
Section C in \cite{AFLV} provides examples of systems for which one
gets exponential decay of correlations (and hence $\xi(n)$ is
summable) and against $\psi\in L^1(\mu)$: this last assumption
requires the density to be bounded from below.
 In conclusion the result of  our note  extends the previous one in two directions : first, we obtain a LDP, and not only an upper
 bound; secondly,  we do not have to assume anything about the
 density, but the fact that it is in the $L^1$ norm with respect to
 the conformal (reference) measure.
\end{enumerate}
\end{Rem}
\section{Proofs}
We begin by observing that the existence of the variance $\sigma^2$
results from a straightforward computation : since the sequence
$\int \phi(\phi \circ T^n) \, d\mu$ decays exponentially fast, it is
absolutely summable, and then we see, by expanding the term $S_n^2$,
that the limit $\sigma^2 = \lim_{n \to \infty} \int
(\frac{S_n}{\sqrt{n}})^2 \, d \mu$ exists and we have $$\sigma^2 =
\int \phi^2 \, d \mu + 2 \sum_{n=1}^{+ \infty} \int \phi (\phi \circ
T^n) \, d \mu$$ We assume from now $\sigma^2 > 0$. Our proof of the
LDP follows closely \cite{HH} except for a minor modification, which
will be mentioned later. The same approach had been employed in
\cite{RBY}. Let $f \in \mathcal{B}$ the density of the measure $\nu$ with respect to $m$. We will apply Gartner-Ellis theorem \cite{DZ, El}, so we
are interested in the convergence of the sequence $\frac{1}{n} \log
\int e^{\theta S_n} f \, dm$ for $\theta \in \mathbb{R}$ small
enough. We introduce the "Laplace transform" operators $P_z$, for $z
\in \mathbb{C}$, defined by $$P_z(f) = P(e^{z \phi}f), \: f \in
\mathcal{B}$$ Assuming for a moment that $P_z$ is well defined, we
see immediately that we have $\int e^{\theta S_n} f \, dm = \int
P^n_{\theta}(f) \, dm$. In order to prove that $P_z$ is a bounded
operator on $\mathcal{B}$, we just have to check that $e^{z \phi}
\in \mathcal{B}$. Since $\mathcal{B}$ is a Banach algebra, the
sequence $\sum_{k=0}^n \frac{(z\phi)^k}{k!}$ converges in
$\mathcal{B}$, and hence in $L^1(m)$. On the other hand, this
sequence converges uniformly, and hence in $L^1(m)$, to $e^{z
\phi}$, and so we get that $$e^{z \phi} = \sum_{n=0}^{+\infty}
\frac{(z \phi)^n}{n!}$$ in $\mathcal{B}$. It also proves that the
map $z \to P_z$ is holomorphic and we have the expansion $$P_z =
\sum_{n=0}^{+\infty}\frac{C_n}{n!} z^n$$ where $C_n(f) = P(\phi^n
f)$. \\ We can now apply perturbation theory for linear operator to
prove the following result. The proof relies on analytic functions
of operators, see \cite{DS}, or on the implicit function theorem,
see \cite{HH}. For $\theta > 0$, we denote $\mathbb{D}_{\theta} = \{
z \in \mathbb{C} \, / \, |z| < \theta \}$.

\begin{Pro} \em

There exist $\theta_0 > 0$, $C > 0$, $\eta_1, \eta_2 > 0$ and
holomorphic functions $\lambda(.) : \mathbb{D}_{\theta_0} \to
\mathbb{C}$, $v(.) : \mathbb{D}_{\theta_0} \to \mathcal{B}$,
$\varphi(.) : \mathbb{D}_{\theta_0} \to \mathcal{B}^{\star}$ and
$Q(.) : \mathbb{D}_{\theta_0} \to \mathcal{L}(\mathcal{B})$ such
that for all $z \in \mathbb{D}_{\theta_0}$ \\ (i) $ \lambda(0) = 1,
v(0) = v, \varphi(0) = m, Q(0) = Q$; \\ (ii) $P_z(f) = \lambda(z)
<\varphi(z), f> v(z) + Q(z) f$ for all $f \in \mathcal{B}$; \\ (iii)
$<\varphi(z), v(z)> = 1$; \\ (iv) $Q(z)v(z) = 0$ and $\varphi(z)
Q(z) = 0$; \\ (v) $|\lambda(z)| > 1 - \eta_1$; \\ (vi) $||Q(z)^n||
\le C (1 - \eta_1 - \eta_2)^n$.
\end{Pro}
So, for all $n \ge 1$, we have $$P_z^n(f) = \lambda(z)^n
<\varphi(z), f> v(z) + Q(z)^n f$$ We can say much more on
eigenvalues and eigenvectors when $z = \theta$ is real. At this
point, we need to show that for every positive  function $f \in
\mathcal{B}$ with $f \neq 0$, there exists a positive linear form
$\varphi \in \mathcal{B}^{\star}$ such that $<\varphi, f> \, > 0$.
In the context of \cite{HH}, since functions are defined everywhere,
there exists $x \in X$ such that $f(x) > 0$, and so the Dirac mass
$\delta_x$ does the job. In our context, Dirac masses are not
available, but the reference measure is usable , since necessarily
$<m,f>
> 0$, otherwise, $f$ would be $0$ $m$-ae, and so $f = 0$ in
$\mathcal{B}$. This was not the case in \cite{HH} because they
consider functions defined everywhere, and not classes of
equivalence. We can also use arguments from complex Banach lattice
theory \cite{MeN} : a modification of the Hahn-Banach theorem shows
that there exists a positive bounded linear form $\varphi$ on
$\mathcal{B}_{\mathbb{R}} = \{ f \in \mathcal{B} \, / \, f(x) \in
\mathbb{R} \; m{\rm -ae}\}$, such that $< \varphi, f> = 1$, and then
we can extend it on all $\mathcal{B}$. This argument could be
employed in more abstract contexts, where the Banach space
$\mathcal{B}$ consists of distributions-like objects and when we
don't have a good knowledge of its topological dual.

\begin{Pro} \em

There exists $0 < \theta_1 < \theta_0$ such that for every $\theta
\in \mathbb{R}$ with $|\theta| < \theta_1$, we have $\lambda(\theta)
> 0$. Furthermore, $v(.)$ and $\varphi(.)$ can be redefined such
that $v(\theta) \ge 0$, $\varphi(\theta) \ge 0$.
\end{Pro}

\begin{proof}
As $P_{\theta}$ is a real operator, we have $P_{\theta} \overline{f}
= \overline{P_{\theta}f}$ for all $f \in \mathcal{B}$. So, we have
$P_{\theta} \overline{v(\theta)} = \overline{P_{\theta} v(\theta)} =
\overline{\lambda(\theta)} \, \overline{v(\theta)}$. Since
$\lambda(\theta)$ is the unique eigenvalue of $P_{\theta}$ with
maximal modulus, we get $\overline{\lambda(\theta)} =
\lambda(\theta)$, and hence $\lambda(\theta) \in \mathbb{R}$. Since
$\lambda(0) = 1$, by a continuity argument, we obtain
$\lambda(\theta) > 0$ for small $\theta$. For $z \in \mathbb{C}$
small enough, $<\varphi(z), \mathds{1}> \neq 0$. We define
$\tilde{v}(z) = <\varphi(z), \mathds{1}> v(z)$ and
$\tilde{\varphi}(z) = <\varphi(z), \mathds{1}>^{-1} \varphi(z)$.
Those new eigenfunctions satisfy obviously the conclusions of the
previous proposition. We have just to prove that $\tilde{v}(\theta)$
and $\tilde{\varphi}(\theta)$ are positive for $\theta \in
\mathbb{R}$ small enough. By the spectral decomposition of
$P_{\theta}$, we see that $\lambda(\theta)^{-n}P_{\theta}^n
\mathds{1}$ goes to $\tilde{v}(\theta)$ in $\mathcal{B}$, and hence
in $L^1(m)$. We then get $\tilde{v}(\theta) \ge 0$ because
$P_{\theta}$ is a positive operator and $\lambda(\theta)$ is
positive too. Now, let $\psi(\theta) \in \mathcal{B}^{\star}$
positive such that $<\psi(\theta), \tilde{v}(\theta)> = 1$. Then,
$\lambda(\theta)^{-n}(P_{\theta}^{\star})^n \psi(\theta)$ goes to
$<\psi(\theta), v(\theta)> \varphi(\theta) =
\tilde{\varphi}(\theta)$, which proves that
$\tilde{\varphi}(\theta)$ is a positive linear form.
\end{proof}

We denote $$\Lambda(\theta) = \log \lambda(\theta)$$ We then have

\begin{Pro} \em

There exists $0 < \theta_2 < \theta_1$ such that for every $\theta
\in \mathbb{B}$ with $|\theta| < \theta_2$ and every $f \in
\mathcal{B}$ with $f \ge 0$ and $\int f \, dm = 1$, we have
$$\lim_{n \to \infty} \frac{1}{n} \log \int e^{\theta S_n} f \, dm =
\Lambda(\theta)$$
\end{Pro}

\begin{proof}
We have the identity $$\begin{aligned} \int e^{\theta S_n} f \, dm =
<m, P_{\theta}^n(f)> & = \lambda(\theta)^n < \varphi(\theta), f> \,
<m, v(\theta)> + <m, Q(\theta)^n f> \\ & = \lambda(\theta)^n ( <
\varphi(\theta), f> \, <m, v(\theta)> + \lambda(\theta)^{-n} <m,
Q(\theta)^n f>) \end{aligned}$$ All involved quantities are
positive, hence we can write $$\frac{1}{n} \log \int e^{\theta S_n}
f \, dm = \log \lambda(\theta) +\frac{1}{n} \log (<\varphi(\theta),
f> \, <m, v(\theta)> + \lambda(\theta)^{-n} <m, Q(\theta)^n f>)$$
Since $$\lim_{\theta \to 0} <\varphi(\theta), f> \, <m, v(\theta)> =
1$$ and since the spectral radius of $Q(\theta)$ is strictly less
than $\lambda(\theta)$, it's easy to see that for $\theta$ small
enough, we have $$\lim_{n \to \infty} \frac{1}{n} \log
(<\varphi(\theta), f> \, <m, v(\theta)> + \lambda(\theta)^{-n} <m,
Q(\theta)^n f>) = 0$$
\end{proof}

In order to apply Gartner-Ellis theorem, we just have to show that
$\Lambda$ is differentiable function, strictly convex in a
neighborhood of $0$. Since $\lambda$ is real-analytic, $\Lambda$ is
too. Computations from perturbation theory \footnote{See corollaries
III.11 and III.6 in \cite{HH}.} show that $\lambda'(0) = \int \phi
\, d \mu = 0$ and $\lambda''(0) = \sigma^2$, so we have
$\Lambda''(0) = \frac{\lambda''(0)\lambda(0) -
\lambda'(0)^2}{\lambda(0)^2} = \sigma^2 > 0$ and we can now apply the following local version of Gartner-Ellis theorem, whose proof can be found in lemma XIII.2 in \cite{HH} :

\begin{Pro} \em

For all $n \ge 1$, denote by $\mathbb{P}_n$ a probability measure on some measurable space $(\Omega, \mathcal{T})$, by $\mathbb{E}_n$ the corresponding expectation operator and by $S_n$ a real valued random variable. Assume that on some interval $[- \theta_{\Lambda}, \theta_{\Lambda}]$, $\theta_{\Lambda} > 0$, we have $$\lim_{n \to \infty} \frac{1}{n} \log \mathbb{E}_n[\exp(\theta S_n)] = \Lambda(\theta),$$ where $\Lambda$ is a strictly convex continuously differentiable function satisfying $\Lambda'(0) = 0$. \\ Define $\epsilon_+ = \frac{\Lambda(\theta_{\Lambda})}{\theta_{\Lambda}} > 0$, $\epsilon_- = \frac{\Lambda(\theta_{\Lambda})}{\theta_{\Lambda}} < 0$ and $c(\epsilon) = \underset{| \theta | \le \theta_{\Lambda}}{\sup} \{\theta \epsilon - \Lambda(\theta) \}$. Then $c$ is a positive function, strictly convex on $[\epsilon_-, \epsilon_+]$, continuous, vanishing only at $0$, and, for all $0 < \epsilon < \epsilon_0 = \epsilon_+$, we have $$ \lim_{n \to \infty} \frac{1}{n} \log \mathbb{P}_n (S_n > n \epsilon) = - c(\epsilon)$$
\end{Pro}

We now prove our Theorem 2.

\begin{proof}[Central Limit Theorem]
By Levy's continuity theorem, it suffices to show that for all $t
\in \mathbb{R}$ $$\lim_{n \to \infty} \int e^{i t
\frac{S_n}{\sqrt{n}}} f \, dm= e^{- \frac{t^2 \sigma^2}{2}}$$ We
have $$\int e^{it \frac{S_n}{\sqrt{n}}} f \, dm = <m,
P_{\frac{it}{\sqrt{n}}}^n(f)> = \lambda(\frac{it}{\sqrt{n}})^n
<\varphi(\frac{it}{\sqrt{n}}), f> \, <m, v(\frac{it}{\sqrt{n}})> +
<m, Q(\frac{it}{\sqrt{n}})^n f>$$ We just have to prove that
$$\lim_{n \to \infty} \lambda(\frac{it}{\sqrt{n}})^n = e^{-
\frac{t^2 \sigma^2}{2}}$$ But the Taylor's expansion says that in a complex neighborhood of
$0$ $$\lambda(z) = \lambda(0) + \lambda'(0) z +
\frac{\lambda''(0)}{2}z^2 + z^2 \eta(z) = 1 + \frac{\sigma^2 z^2}{2}
+ z^2 \eta(z)$$ where $\lim_{z \to 0} \eta(z) = 0$. Then, a standard
computation concludes the proof.
\end{proof}
\section{Application to uniformly expanding maps}The main application of our Theorem 1
will be   to multidimensional piecewise uniformly expanding maps, in
particular when we equip them with the space of the quasi-H\"older
functions. This space, introduced by Keller \cite{K}, developed by
Blank~\cite{Bl} and successfully applied by Saussol \cite{S} and
successively by Buzzi~\cite{Bu} (see also \cite{BK}) and
Tsujii~\cite{Ts}, reveals to be very useful to control the
oscillations of a function under the iteration of the transfer
operator across the discontinuities of the map. Moreover it verifies
the algebraic assumption 3 in Section 2 above;  it is not
straightforward  to replace this condition and in order to fit with
the Hennion-Herv\'e theory, if one uses the more conventional spaces
of bounded variation functions (see for instance \cite{Ad, BG, C,
Liv2, PGB})
 or the Sobolev spaces \cite{Tho}, and this topic deserves to be
investigated in the future \footnote{In fact, if we check the proof,
we only need the fact that
 $\mathcal{B}$ is Banach algebra and $\phi \in \mathcal{B}$ to prove that the operators $P_z$ are well defined and
 holomorphic in $z$. So we can suppose the weaker assumption that $\phi$ is such that $P_z$ define a holomorphic family
  of bounded operators on $\mathcal{B}$ for $z$ in a complex neighbourhood of $0$.}. \\

Let us now recall the precise definitions of our system by following
closely the assumptions imposed in \cite{S}. Let $M \subset
\mathbb{R}^d$ be a compact subset with $\overline{{\rm int} M} = M$
and piecewise $C^1$ boundary. We denote by $d$ the Euclidean
distance and by $m$ the Lebesgue measure on $\mathbb{R}^d$. We can
assume without loss of generality that $m(M) = 1$. For $A \subset M$
and $\epsilon > 0$, we denote $B_{\epsilon}(A) = \{x \in
\mathbb{R}^d \, / \, d(x, A) \le \epsilon \}$. Let $T : M \to M$ a
measurable application, and suppose there exists $0 < \alpha \le 1$
such that for some small enough $\epsilon_0$ we have :

\begin{enumerate}
\item There are finitely many disjoint open sets $U_i \subset M$ with $m(M \setminus \cup_{i} U_i) = 0$ such that for each $i$, $T_i := T \vert_{U_i} \to M$ is $C^{1+\alpha}$ and can be extended on a neighborhood $V_i$ of $U_i$ to a $C^{1+\alpha}$ map $T_i : V_i \to \mathbb{R}^d$ such that $B_{\epsilon_0}(T_iU_i) \subset T_i(V_i)$. Moreover, each $T_i : V_i \to \mathbb{R}^d$ is injective with $C^{1+\alpha}$ inverse;
\item There exists $c > 0$ such that for any $i$, and any $x,y \in T(U_i)$ with $d(x,y) \le \epsilon_0$ we have $$\vert \det DT_i^{-1}(x) - \det DT_i^{-1}(y) \vert \le c \vert \det DT_i^{-1}(x) \vert d(x,y)^{\alpha} ;$$
\item There exists $s(T) < 1$ such that $$\sup_i \sup_{x \in T_i (V_i)} || DT_i^{-1}(x) || < s(T);$$
\item Boundaries of $U_i$ are piecewise $C^1$ codimension one embedded compact submanifolds and we have $\eta_0(T) < 1$ where $$\eta_0(T) = s(T)^{\alpha} + \frac{4s(T)}{1-s(T)}Y(T) \frac{\gamma_{d-1}}{\gamma_d}$$ $$Y(T) = \sup_{x \in \mathbb{R}^d} \sum_i \sharp \{ {\rm smooth \; pieces \; intersecting \; } \partial U_i {\rm \; and \; containing \; } x \}$$ and $\gamma_d = \frac{\pi^{d/2}}{(d/2)!}$ is the $d$-volume of the $d$-dimensional unit ball of $\mathbb{R}^d$.
\end{enumerate}
The last condition can be greatly weakened, but the condition in
\cite{S} is of a very abstract nature, and it's more easy to handle
with this one when the boundaries of the $U_i$ are smooth. We define
then the functional space on which acts the transfer operator. Let
$f \in L^1(\mathbb{R}^d)$. If $A \subset \mathbb{R}^d$ is a Borel
subset, we define the oscillation of $f$ over $A$ by $${\rm
osc}(f,A) = \underset{x_1, x_2 \in A}{\rm ess \, sup} \, |f(x_1) -
f(x_2)|$$ where the essential supremum is taken with respect to the
product measure $m \times m$ on $A \times A$. We get a lower
semi-continuous and hence measurable function $x \to {\rm osc}(f,
B_{\epsilon}(x))$. We set $$|f|_{\alpha} = \sup_{0 < \epsilon \le
\epsilon_0} \frac{1}{\epsilon^{\alpha}} \int_{\mathbb{R}^d} {\rm
osc}(f, B_{\epsilon}(x)) dx$$ We define $$V_{\alpha}(\mathbb{R}^d) =
\{f \in L^1(\mathbb{R}^d) \, / \, |f|_{\alpha} < \infty \}$$ and
$$V_{\alpha}(M) = \{f \in V_{\alpha}(\mathbb{R}^d) \, / \, {\rm
supp} \, f \subset M \}$$ both endowed with the norm $||f||_{\alpha}
= ||f||_{L^1_m} + |f|_{\alpha}$. Adapting proofs from \cite{K}, we
can show that $V_{\alpha}(M)$ is Banach space, with compact
injection in $L^1(M)$. It's proven in \cite{S} that $V_{\alpha}(M)$
is also a Banach algebra, and it's obviously a Banach lattice. So,
if we want to apply our previous results to those maps, we are left
to prove that the transfer operator for $T$ acts on $V_{\alpha}(M)$
and is quasi-compact of diagonal type. But Saussol proved in this
context a Lasota-Yorke inequality (lemma 4.1 in \cite{S}) which
implies the quasi-compactness of the Perron-Frobenius operator.
Hence, assuming that the system is mixing, we get the central limit
theorem and the large deviations principle for bounded real
observables in $V_{\alpha}(M)$.

\end{document}